\documentclass[letterpaper, 10 pt, conference]{ieeeconf}  

\IEEEoverridecommandlockouts                              
\usepackage{times} 
\usepackage{amsmath,amsfonts} 
\usepackage{amssymb}  
\usepackage{tabularx}
\usepackage[normalem]{ulem}
\usepackage{dsfont}
\usepackage{graphicx}
\usepackage{subcaption}
\usepackage[dvipsnames,usenames]{xcolor}
\usepackage[colorlinks,%
linkcolor=BrickRed,%
filecolor=BrickRed,%
citecolor=RoyalPurple,%
]{hyperref}
\usepackage{algorithm}
\usepackage{algorithmic}
\usepackage{float}

\DeclareFontEncoding{LS1}{}{}
\DeclareFontSubstitution{LS1}{stix}{m}{n}
\DeclareSymbolFont{stixletters}{LS1}{stix}{m}{it}
\DeclareMathAccent{\cev}{\mathord}{stixletters}{"91}

\def\Re{\mathbb{R}}
\def\R{\mathbb{R}}

\def\argmin{\mathop{\text{\rm arg\,min}}}

\def\notes#1{\marginpar{\tiny #1}\typeout{Notes!
Notes!
Notes!
}}
\renewcommand{\notes}[1]{\typeout{notes!}}

\def\Re{\field{R}}

\def\Expect{{\mathbb E}}

\def\R{\mathbb{R}}

\newtheorem{remark}{Remark}
\newtheorem{proposition}{Proposition}

\newtheorem{assumption}{Assumption}
\newtheorem{problem}{Problem}

\def\beq{\begin{eqnarray}} 
\def\bc{\begin{center}} 
\def\be{\begin{enumerate}}
\def\bi{\begin{itemize}} 
\def\bs{\begin{small}}
\def\bS{\begin{slide}}
\def\ec{\end{center}} 
\def\ee{\end{enumerate}}
\def\ei{\end{itemize}}
\def\es{\end{small}}
\def\eS{\end{slide}}
\def\eeq{\end{eqnarray}}

\newcommand{\trace}{\text{Tr}}

\newcommand{\ud}{\,\mathrm{d}}

\def\Re{\mathbb{R}}

\def\argmin{\mathop{\text{\rm arg\,min}}}

\renewcommand{\Re}{\mathbb{R}}

\newcounter{rmnum}

\newcounter{anum}

\title{\LARGE \bf A Time-Reversal Control Synthesis for Steering the State of \\Stochastic Systems
}

 \author{Yuhang Mei$^{\star}$, Amirhossein Taghvaei$^{\star}$, Ali Pakniyat$^{\dagger}$
 \thanks{Yuhang Mei and Amirhossein Taghvaei are supported by the National Science Foundation (NSF) award EPCN-2318977 and EPCN-2347358.}
     \thanks{$^\star$Department of Aeronautics \& Astronautics, University of Washington, Seattle; {\tt\small yuhangm@uw.edu,amirtag@uw.edu}.}   
	\thanks{$^\dagger$Department of Mechanical Engineering,
University of Alabama; {\tt\small apakniyat@ua.edu}.}
}

\begin{document}

      \maketitle
	 \thispagestyle{empty}
	 \pagestyle{empty}


	 \begin{abstract}
          This paper presents a novel approach for steering the state of a stochastic control-affine nonlinear system to a desired target within a finite time horizon. Our method leverages the time-reversal of diffusion processes to construct the required feedback control law. Specifically, the control law is the so-called score function associated with the time-reversal of random state trajectories that are initialized at the target state and are simulated backwards in time.  A neural network is trained to approximate the score function, enabling applicability to both linear and nonlinear stochastic systems.  Numerical experiments demonstrate the effectiveness of the proposed method across several benchmark examples. 

	\end{abstract}

    \section{Introduction}
    Steering the state of a stochastic system to a target state or distribution is a fundamental problem in stochastic control and useful in applications such as stochastic thermodynamics~\cite{peliti2021stochastic,chen2019stochastic,fu2021maximal}, machine learning~\cite{chen2023generative,domingo2024stochastic,rapakoulias2024go}, and robotics~\cite{shah2012stochastic,gorodetsky2018high,theodorou2011iterative,okamoto2019optimal,liu2024optimal}. 
One prominent framework for studying this problem is the  theory of Schr\"odinger bridges for diffusion processes~\cite{wakolbinger1990schrodinger,pavon1991free}, which aims to find an optimal control law that drives the system from a given initial distribution to a specified target distribution over a finite horizon. Exact solutions are available for linear stochastic systems with Gaussian initial and target distributions~\cite{chen2015optimal,chen2015part2,teter2024schr,teter2025markov}, and numerical procedures extend these results to non-Gaussian~\cite{liu2023generalized,mei2024flow} and nonlinear drift~\cite{caluya2021wasserstein} cases. 
Beyond Schr\"odinger bridge theory, alternative approaches using the stochastic maximum principle and convex duality have been proposed to derive optimal control policies that steer the state toward desired target  points~\cite{pakniyat2021steering, pakniyat2021partially} or distributions~\cite{pakniyat2022convex, pakniyat2024distributionally}.
{However, the reliance on optimality in both these and Schr\"odinger-based approaches makes controller identification for nonlinear settings computationally expensive, as it requires solving complex partial differential equations.}

In contrast, this paper shifts attention away from optimality (and initial-state dependence) toward designing a feedback law that guarantees finite-time attractability of a chosen target. Specifically, we consider a stochastic process $X$ that is governed by a control-affine nonlinear stochastic system~\eqref{eq:sde} and propose a framework for deriving feedback control laws that steer the state $X_t$ toward a target state $X_T=x_f$, at a finite time horizon $T$, without requiring optimality relative to the initial condition.

\begin{figure}[t]
\centering
\includegraphics[width=0.98\hsize]{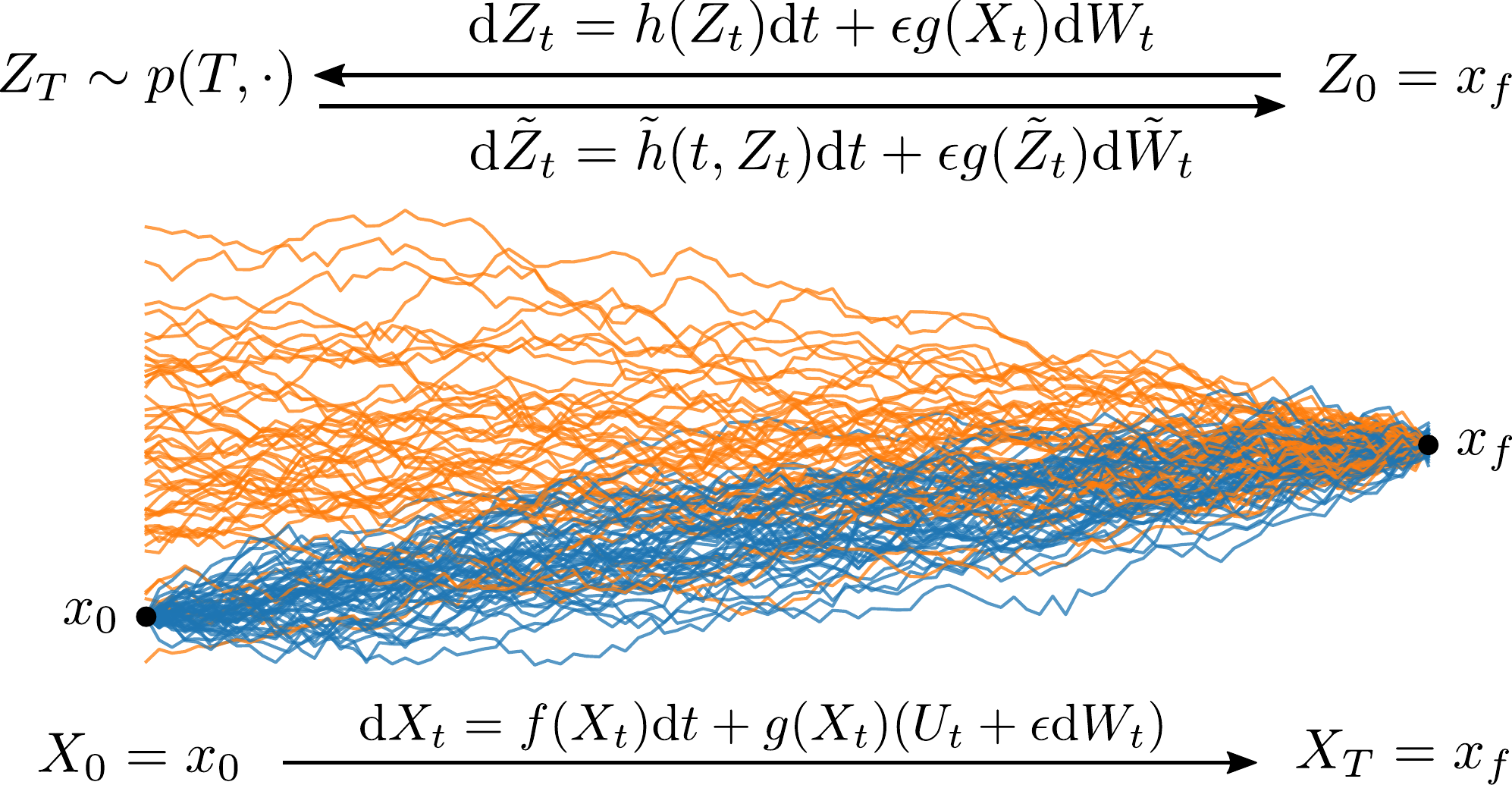}
\caption{Illustration of the proposed time-reversal methodology to steer the state $X_t$ from $x_0$ to $x_f$. The auxiliary process $Z_t$ is simulated backward in time from $x_f$. The dynamics of the time-reversal $\tilde Z_t:=Z_{T-t}$  is forward in time, with an absorption property to $x_f$. This dynamics is used to design the control law $U_t=k(t,X_t)$. The expressions for  the functions $h$, $\tilde h$, and $k$ appear in Section~\ref{sec:methodology}.}
\label{fig:schematic}
\end{figure}

Our approach is inspired by time-reversal theory of diffusions~\cite{anderson1982reverse,haussmann1986time,cattiaux2023time}, which has recently gained attention in machine learning through diffusion-based generative models for images~\cite{NEURIPS2020_4c5bcfec,song2020denoising,song2019generative,songscore2021score}. In these models, one first simulates a stochastic differential equation (SDE) that gradually adds noise to transform complex data (e.g., images) into a simpler, typically Gaussian distribution. The trajectory of this noising process is then used to learn the so-called score function, which in turn is used to run the SDE in reverse. Sampling from the Gaussian and applying this reverse process (the ``denoising'' procedure) generates new samples resembling the original data. From a control-theoretic viewpoint, the learned score function serves as a control law that transforms the Gaussian distribution to the target data distribution. A key benefit of this procedure is its computational tractability as it involves solving a regression problem rather than resorting to dynamic programming or the maximum principle, though at the cost of losing optimality.

Inspired by these diffusion generative models, we draw an analogy for our setting by simulating the state of stochastic system backward in time from a given target state $x_f$ (or a small Gaussian distribution around it). Reversing this backward process in the forward-time direction ensures convergence to the target, with the learned score function serving as the feedback control law. To provide an intuitive understanding of our approach, we outline the key steps involved in constructing the proposed feedback control law, accompanied with an illustration in Figure~\ref{fig:schematic}. 
\begin{enumerate}
    \item \textbf{Auxiliary Process $\boldsymbol{Z}$:} We construct an auxiliary process $Z$, initialized at the desired terminal state $x_f$. This process generates a probability density function $p(t,x)$ with a delta distribution at the desired terminal location at its initial time, i.e., $p(0,x) = \delta_{x_f}(x)$ while producing spread-out distributions at other times. This serves as the foundation for defining a time-reversed process.

    \item \textbf{Time-Reversed Auxiliary Process $\boldsymbol{\tilde{Z}}$}: By reversing time for the auxiliary process $Z$, we obtain a new process $\tilde{Z}_t =Z_{T-t}$ with the probability density function $\tilde{p}(t,x) = p(T-t,x)$ possessing a delta distribution at the desired terminal location (matching the target state) at its terminal time, i.e., $\tilde{p}(T,x) = \delta_{x_f}(x)$. The governing dynamics of $\tilde{Z}$ possess an absorption property for the desired terminal state, i.e., almost all sample paths are attracted to and absorbed by the desired state. This property is critical for ensuring that trajectories converge to the target state.

    \item \textbf{Construction of Feedback Control Law for $\boldsymbol{X}$:} The feedback law synthesis for $X$ leverages the absorption property of the time-reversed process $\tilde{Z}$, to yield closed loop dynamics that enforce almost sure convergence to the target state within a fixed time horizon. Crucially, if the initial condition $x_0$ of the original process $X$ lies within the support of the initial distribution of $\tilde{Z}$, then $x_0$ belongs to the fixed-time stochastic region of attraction of the desired terminal state almost surely.
\end{enumerate}

\noindent
{\bf Relation to diffusion bridges:} A diffusion bridge is defined as the law of the diffusion processes $X$ conditioned on the event that $X_0=x_0$ and $X_T=x_f$. Using Doob's $h$-transform \cite[Sec. 7.5]{sarkka2019applied}, the conditioned process satisfies the same SDE as the original diffusion but with an added drift that can be interpreted as a control law steering the state to the target. However, this control law is distinct from the one we propose in this paper. In particular, our control law is not generally of gradient form, whereas the Doob $h$-transform yields the gradient of the log-transition density.  Interestingly, the two control laws coincide in the linear Gaussian setting, suggesting a deeper connection between the approaches. From a computational standpoint, there is a recent work~\cite{10.1093/biomet/asaf048} that applies time-reversal and score matching to simulate diffusion bridges. However, their proposed method  requires applying time reversal and score-function approximation twice~\cite[Sec. 2.3]{10.1093/biomet/asaf048}, with potential error accumulation, while our method requires only a single such operation.
A detailed numerical comparison to fully characterize the trade-offs is planned for future investigation.

The paper is organized as follows. Section \ref{sec:problem_formulation} presents the problem formulation and the necessary background on the time-reversal theory of diffusions. Section \ref{sec:methodology} presents our proposed  time-reversal methodology for control synthesis, along with the analysis of the linear Gaussian setting. Section \ref{sec:results} presents numerical experiments on several benchmark examples, demonstrating the performance of the proposed method in both linear and nonlinear settings. 


\subsection{Notation}
The notation $\frac{\partial}{\partial x}$ is used to denote the derivative with respect to the $x$ variable. For example, for a smooth function $f:\R^n \to \Re$,  $\frac{\partial f}{\partial x}:\R^n\to \R^n$ denotes the gradient of $f$ with respect to $x$. And for a smooth vector-field $f:\R^n \to \Re^n$,  $\frac{\partial f}{\partial x}:\R^n\to \R^{n\times n}$ denotes the Jacobian.  The probability density function (PDF) of a multivariate Gaussian distribution with mean vector  $m$ and covariance matrix $\Sigma$ is denoted by $\mathcal{N}(x;\mu,\Sigma)$. For a positive definite matrix $P$, the weighted $2$-norm of a vector $x \in \R^n$, is denoted by $\|x\|_P$, and defined as $\|x\|_P := \sqrt{x^\top Px}$.

    \section{Problem Formulation and Background}\label{sec:problem_formulation}
    \subsection{Problem setup}
Consider a  control system governed by a control-affine stochastic differential equation (SDE)
\begin{equation}\label{eq:sde}
    \ud X_t = f(X_t)\ud t + g(X_t) (U_t \ud t + \epsilon \ud W_t), ~~X_0 = x_0,
\end{equation}
where $X_t \in \mathbb{R}^n$ is the state, $U_t \in \R^m$ is the control input, $W_t \in \R^m$ is the standard Wiener process that denotes the process noise, $f:\R^n\to\R^n$ and $g:\Re^n \to \Re^{n\times m}$ are drift and diffusion functions, and $\epsilon >0$ is a positive parameter that denotes the strength of the noise. Let $\mathcal{F}_t = \sigma(W_s; 0\leq s \leq t)$ be the filtration generated by the Wiener process. The control input $U_t$ is constrained to be adapted to the filtration $\mathcal{F}_t$. We are interested in the problem of designing the control input that steers the state of the system to a given target state. 


\begin{problem}\label{prob:steer-to-y}
    Given a target state $x_f \in  \R^n$ and a fixed terminal time $T>0$, find a feedback control law $\{U_t=k(t,X_t);t\in[0,T]\}$, for a function $k: [0,T] \times \R^n \rightarrow \R^m$, such that  $X_T = x_f$ almost surely.
\end{problem}

Note that the $\mathcal{F}_t$-adaptability condition is automatically satisfied when $U_t$ is expressed as a function of $X_t$.  We also consider a relaxed version of the problem, where the  equality is relaxed to a bound on the average distance from the target. This relaxation allows us to obtain control laws that are not singular as $t\to T$. See Section~\ref{sec:singular}.

\begin{problem}\label{prob:approx-steer-to-y}
    Given a target state $x_f \in  \R^n$, a fixed terminal time $T>0$, and error tolerance $\delta>0$, find a feedback control law $\{U_t=k(t,X_t);t\in[0,T]\}$, for a function $k: [0,T] \times \R^n \rightarrow \R^m$, such that  $\mathbb E[\|X_T - x_f\|^2]\leq \delta$.
\end{problem} 

\begin{remark}[Modelling assumptions]
   The control-affine structure of the dynamic model in~\eqref{eq:sde} is critical for the applicability of the proposed method. This model is common in robotics, aerospace, and finance \cite{zhou2023safe,liu2019fuel,calafiore2008multi}. The time-invariant assumption for the functions $f$ and $g$ is made for ease of presentation and can be relaxed to be time-varying. 
\end{remark}


Our solution methodology for these problems is based on the time-reversal of diffusion theory, which we review  next. 

\subsection{Time-Reversal of Diffusions}
\label{sec:tr}
Consider a diffusion process $Z \mathrel{:=} \{Z_t\in \R^n;0 \leq t \leq T\}$ governed by the following SDE 
\begin{equation}\label{eq:fsde}
    \ud Z_t \!=\! h(Z_t) \ud t + \epsilon g(Z_t) \ud W_t,\quad Z_0=z,
    \end{equation}
    where   $h:  \R^n \rightarrow \R^n$ is a drift function, and $z$ is an initial condition, that are later designed as part of our methodology, in Section~\ref{sec:naive-alg}.  We make the following assumption about the functions $h$ and $g$. 
    
\begin{assumption}\label{asum-fg}
    The functions $h$ and $g$ are smooth and globally Lipschitz. 
\end{assumption}

Assumption \ref{asum-fg} implies that SDE \eqref{eq:fsde} admits a unique strong solution \cite[Thm. 5.2.1.]{oksendal2003stochastic}. This allows us to define the time-reversal of $Z_t$ according to 
\begin{equation*}
    \tilde{Z} \mathrel{:=} \{\tilde{Z}_t =Z_{T-t} ;\, 0 \leq t \leq T\}. 
\end{equation*} 
The time-reversal theory is concerned with obtaining the SDE for the reversed process $\tilde Z_t$. In order to do so, we follow the approach presented in~\cite{haussmann1986time} which, in addition to Assumption~\ref{asum-fg}, requires a suitable integrability condition on the density of $Z_t$. In particular, letting $p(t,x)$ denote the probability density function of $Z_t$, it is required that 
    \begin{equation}\label{eq:asum-bound}
        \int_{t_0}^T \int_\Omega |p(t,x)|^2 + \|g(x)^\top \frac{\partial p}{\partial x} (t,x)\|^2 \ud x \ud t < \infty ,
    \end{equation}
    for any open bounded set $\Omega \subset \Re^n$  and $t_0>0$.
This condition is valid under  the following assumption about the drift and diffusion functions. 


\begin{assumption}\label{asum-brackets}
    For all $x\in \R^n$, the subspace generated by the Lie-brackets\footnote{The Lie-bracket of two smooth vector-fields $g_i:\R^n \to \R^n$ and $g_j:\R^n\to \R^n$ is defined as $[g_i,g_j](x)=\frac{\partial g_j}{\partial x}(x)g_i(x) - \frac{\partial g_i}{\partial x}(x)g_j(x)$} of the elements of $\{h(x),g_1(x),\ldots,g_m(x)\}$, with $g_i$ the $i$-th column of~$g$, spans the entire space $\R^n$. 
\end{assumption}

Assumption \ref{asum-brackets} is knwon as the H\"ormander condition that ensures the generator associated with the diffusion process~\eqref{eq:fsde} is hypoelliptic, implying that $Z_t$ admits a smooth density $p(t,x)$ for any $t>0$~\cite{chaleyat1984hypoellipticity,hormander1967hypoelliptic}. As a result, condition~\eqref{eq:asum-bound} is satisfied. With a smooth density, one can define the so-called score function $s:(0,T] \times \Re^n \to \Re^n$ whose $i$-th component is given by
  \begin{align}\label{eq:score-def}
            s_i(t,x) \mathrel{:=} \frac{1}{p(t,x)} \sum_{j=1}^n  \frac{\partial}{\partial x_j}(G_{i,j}(x)p(t,x)),
        \end{align}
        where $G_{i,j}(x)$ is the $(i,j)$-entry of the matrix $G(x):=g(x)g(x)^\top \in \R^{n\times n}$. 
Then, according to \cite[Thm. 2.1]{haussmann1986time},
the reversed process $\tilde Z_t$ satisfies the SDE
\begin{align}
            \ud \tilde{Z}_t = -h(\tilde Z_t)\ud t +  \epsilon^2 s(T-t,\tilde Z_t) \ud t+ \epsilon g(\tilde{Z}_t) \ud  \tilde W_t,    \label{eq:tr-sde} 
        \end{align}
        where $\tilde W_t \in \R^m$ is the standard $m$-dimensional Wiener process. Note that the reversed process satisfies the condition $\tilde Z_T=Z_0 = z$. Therefore, the dynamics generated by the score function and  the drift term has an absorption property for the state $z$ at terminal time $T$. This is the basis for our methodology in Section~\ref{sec:naive-alg}. 
      
         \begin{remark}
        The SDE for the time-reversal process may also be obtained using Girsanov theorem, as in~\cite{cattiaux2023time}, but this alternative approach requires the diffusion function to be non-singular, i.e. $G(x):=g(x)g(x)^\top$ be uniformly positive-definite for all $x\in \R^n$~\cite{anderson1982reverse,cattiaux2023time}. This assumption is restrictive due to the role that the function $g$ plays in the stochastic control problem~\eqref{eq:sde}. In particular, the columns of $g$ are the directions that the control input affects the state. Therefore, a positive-definite assumption on $G$ implies full control authority which is restrictive in many control applications. Assumption~\ref{asum-brackets} is a weaker assumption that allows the construction of the time-reversal SDE for a degenerate diffusion. It is also in agreement with the local accessibility condition in geometric control theory  for the deterministic version of the system~\eqref{eq:sde},  when $\epsilon=0$~\cite{krener1974generalization,sussmann1987general}. 
    \end{remark}
    
    The formula for the score function~\eqref{eq:score-def} depends on the density $p(t,x)$ which is explicitly available only in the special linear Gaussian setting~(see Section~\ref{sec:lin-gau}). In a general nonlinear and non-Gaussian setting, the score function  is approximated as the solution to a stochastic optimization problem, as described in the next subsection. 
\subsection{Score function approximation}
It is numerically useful to note that the score function is the solution to the minimization problem $\min_\psi\,J(\psi)$ where 
 \begin{align}\label{eq:score-opt}
       J(\psi):= \mathbb{E}[ \frac{1}{2}\|\psi(t,Z_t)\|^2 + \sum_{i,j=1}^n G_{i,j}(Z_t) \frac{\partial \psi_i}{\partial x_j} (t,Z_t)) ],
    \end{align}
    and the expectation is both over $t \sim \text{Unif}[0,T]$ and $Z_t$, solution to~\eqref{eq:fsde}. 
    This follows by writing the expectation as the integral with respect to the density $p(t,x)$, application of integration by parts on the second term,  and expressing the objective function~\eqref{eq:score-opt} as 
    \begin{align*}
        \mathbb{E}\left[ \frac{1}{2}\|\psi(t,Z_t)-s(t,Z_t)\|^2   \right] + \text{(constant)},
    \end{align*}
    where the constant is independent of $\psi$~\cite[Thm. 1]{hyvarinen2005estimation}. The optimization \eqref{eq:score-opt} is known as implicit score matching \cite{tang2024score}. This optimization procedure is later used in the construction of our numerical algorithm in Section~\ref{sec:Algorithm}.  


    \subsection{Probability transition kernels}\label{sec:kernel}
   This section introduces notations and definitions for probability transition kernels associated with SDEs~\eqref{eq:fsde} and~\eqref{eq:tr-sde} that are later used in the proof of Propostions~\ref{prop:sol2p1} and~\ref{prop:terminal-error}. 
    Let $\kappa_{t,s}(x'|x)$ denote the probability transition kernel from time $s$ to time $t$ that is associated with SDE~\eqref{eq:fsde},
i.e. the conditional probability density function of $Z_t=x'$ given $Z_s=x$, for any $t\geq s>0$. Similarly, let $\tilde \kappa_{t,s}(x'|x)$ denote the probability transition kernel for SDE~\eqref{eq:tr-sde}. Using the probability transition kernel, the joint probability density function of $(Z_t,Z_s)$ satisfies
\begin{align*}
\int\limits_{B_{x} \times B_{x'}} \hspace{-9pt}P_{Z_t,Z_s}(x',x) \ud x^\top \! \ud x'\! =\! \int\limits_{B_{x} \times B_{x'}} \hspace{-9pt}\kappa_{t,s}(x'|x) p(s,x) \ud x^\top \!\ud x'
\end{align*}
for arbitrary Borel sets $B_{x}, B_{x'} \subset \R^n$, implying   
\begin{align*}
     P_{Z_t,Z_s}(x',x) =  \kappa_{t,s}(x'|x) p(s,x).
\end{align*}
Similarly, for $(\tilde Z_t,\tilde Z_s)$ we have 
\begin{align*}
    P_{\tilde Z_t,\tilde Z_s}(x',x) =  \tilde \kappa_{t,s}(x'|x) p(T-s,x), 
\end{align*}
where we used the fact that $\tilde Z_s = Z_{T-s}$ with probability density $p(T-s,x)$. The identity $(Z_T,Z_{T-t})= (\tilde Z_0,\tilde Z_t)$ implies the equality $P_{Z_T,Z_{T-t}}(x',x)= P_{\tilde Z_{t},\tilde Z_0}(x,x')$, concluding the relationship 
\begin{align}\label{eq:kappa-identity}
   &\kappa_{T,T-t}(x'|x) p(T-t,x) = \tilde \kappa_{t,0}(x|x') p(T,x')
\end{align}
between the two kernels $\kappa$ and $\tilde \kappa$.

 \section{Proposed Methodology}
 \label{sec:methodology}

\subsection{Solution to problem~\ref{prob:steer-to-y}}\label{sec:naive-alg}

We propose to solve Problem~\ref{prob:steer-to-y} using the time-reversal theory that is presented in Section~\ref{sec:tr}. We start with the process $Z_t$ from~\eqref{eq:fsde} and initialize it  at $Z_0=x_f$. With this initialization, the reversed process $\tilde Z_t:=Z_{T-t}$ satisfies  the terminal condition $\tilde Z_T=Z_0=x_f$, and its dynamics has an absorption property for $x_f$. Therefore, in order to solve Problem~\ref{prob:steer-to-y}, we design $h$ so that SDE~\eqref{eq:tr-sde} for $\tilde Z_t$ takes a form similar to~\eqref{eq:sde} for $X_t$. Using the decomposition of the score function as
\begin{equation}\label{eq:score-decom}
    s(t,x) =  g(x) k^\star(t,x) + \mathfrak g(x)
\end{equation}
where the $i$-th component of $ k^\star:[0,T] \times \R^n \rightarrow \R^m$ and  $\mathfrak g :  \R^n \rightarrow \R^n$ are defined as 
\begin{align}\label{eq:k-star}
     k_i^\star(t,x) &\mathrel{:=} \frac{1}{ p(t,x)} \sum_{j=1}^n \partial_{x_j}(g_{j,i}(x)  p(t,x)),\\
 \mathfrak g_i(t,x) &\mathrel{:=} \sum_{j=1}^n \sum_{k=1}^m g_{j,k}(x) \partial_{x_j}g_{i,k}(x), \label{eq:frak-g}  
\end{align}
the SDE~\eqref{eq:tr-sde} takes the form 
\begin{align*}
    \ud \tilde Z_t = 
    &\big(-h(\tilde Z_t) + \epsilon^2\mathfrak g(\tilde Z_t)\big)\ud t \\&+  g(\tilde Z_t) \big(\epsilon^2 k^\star(T-t,\tilde Z_t)\ud t + \epsilon \ud \tilde W_t\big).
\end{align*}
Designing the function $h$ and control law $k$ according to 
\begin{align}\label{eq:tilde-f}
    h(x) &:= -f(x) +  \epsilon^2 \mathfrak  g(x),\\
    k(t,x) &:= \epsilon^2 k^\star(T-t,x),\label{eq:fbcontrol}
\end{align}
yields the following expressions for SDEs of $Z_t$ and $\tilde Z_t$:
\begin{subequations}
\begin{align}\label{eq:Z-alg}
\ud Z_t &= (-f(Z_t) + \epsilon^2 \mathfrak g(Z_t))\ud t +  \epsilon g(Z_t) \ud W_t,~Z_0=x_f
\\ \ud \tilde Z_t &= 
    f(\tilde Z_t)\ud t+  g(\tilde Z_t) (k(t,\tilde Z_t)\ud t + \epsilon \ud \tilde W_t).\label{eq:tilde-Z-alg}
\end{align}
Moreover, using the control law~\eqref{eq:fbcontrol} in~\eqref{eq:sde} concludes the following SDE for $X_t$: 
\begin{align}\label{eq:X-alg}
    \ud X_t = f(X_t)\ud t\!+ \! g(\tilde X_t) (k(t, X_t)\ud t + \epsilon \ud W_t),\quad X_0=x_0.
\end{align}
\end{subequations}
In summary, we have constructed three stochastic processes: 
\begin{enumerate}
    \item The process $Z_t$ that solves~\eqref{eq:Z-alg} from initial condition $Z_0=x_f$. The density of this process is denoted by $p(t,x)$;
    \item The process $\tilde Z_t := Z_{T-t}$ that solves~\eqref{eq:tilde-Z-alg} and satisfies the condition $\tilde Z_T=Z_0=x_f$. The control law $k(t,x)$ is defined in~\eqref{eq:fbcontrol} and~\eqref{eq:k-star};
    \item The process $X_t$ that solves~\eqref{eq:X-alg} starting from the initial condition $X_0=x_0$.   
\end{enumerate}
Note that, the SDE for $\tilde Z_t$ has the same form as the SDE for $X_t$ and the control law steers the process $\tilde Z_t$ to the target state $\tilde Z_T=x_f$. However, it remains to be shown that the control law steers $X_t$ to $X_T=x_f$, despite the difference in the initial conditions; $\tilde Z_0 = Z_T$ is random with probability density function $p(T,x)$, whereas $X_0=x_0$ is deterministic. The next result shows that $X_t$ reaches the target state $X_T=x_f$  whenever the initial condition satisfies $p(T,x_0)>0$.  

\begin{proposition}\label{prop:sol2p1}
Let $p(t,x)$ denote the probability density function  of $Z_t$ defined according to~\eqref{eq:Z-alg} and define the control law $k(t,x)$ according to~\eqref{eq:fbcontrol} and~\eqref{eq:k-star}. If the initial condition $X_0=x_0$ satisfies $p(T,x_0)>0$, then, problem~\ref{prob:steer-to-y} is solved with the feedback control law $k(t,x)$. 
\end{proposition}
\begin{proof}
Let $\kappa_{t,s}(x'|x)$ and $\tilde \kappa_{t,s}(x'|x)$ be the probability transition kernels for SDEs~\eqref{eq:Z-alg} and~\eqref{eq:tilde-Z-alg}, respectively, as defined in Section~\ref{sec:kernel}.  
The probability transition kernel associated with the SDE~\eqref{eq:X-alg} is also $\tilde \kappa_{t,s}(x'|x)$ due to the fact that SDEs~\eqref{eq:tilde-Z-alg} and~\eqref{eq:X-alg} are identical.  Therefore, with the initial condition $X_0=x_0$, the probability density function of $X_t$ becomes equal to $\tilde \kappa_{t,0}(x|x_0)$. The goal is to show that  $\tilde \kappa_{t,0}(x|x_0)$ approaches the Dirac delta distribution $\delta_{x_f}(x)$ (in the weak sense) as $t$ approaches $T$. The identity~\eqref{eq:kappa-identity} implies  
    \begin{align*}
   \tilde \kappa_{t,0}(x|x_0)  &= \frac{  \kappa_{T,T-t}(x
    _0|x)p(T-t,x)}{ p(T,x_0)}\\
    &=\frac{  \kappa_{T,T-t}(x
    _0|x)p(T-t,x)}{ \kappa_{T,0}(x_0|x_f)}
\end{align*}
where we used the assumption that $p(T,x_0)>0$ and the fact that $ p(T,x)= \kappa_{T,0}(x|x_f)$ due the initial condition $Z_0=x_f$. 
 Taking the limit as $t\to T$ and using the fact that $p(T-t,x)$ approaches $\delta_{x_f}(x)$ concludes the result. 
\end{proof}

\subsection{Avoiding singularity of the control law}\label{sec:singular}
The feedback control law~\eqref{eq:fbcontrol} becomes singular in the limit as $t$ approaches the terminal time $T$ because the distribution $p(T-t,x)$ approaches the Dirac delta distribution $\delta_{x_f}(x)$. The singularity is unavoidable when an almost sure constraint $X_T=x_f$ is required. In order to avoid the singularity, we consider Problem~\ref{prob:approx-steer-to-y} where the almost sure equality is relaxed to a bound on the average distance to the target. We propose to solve problem~\ref{prob:approx-steer-to-y} using the time-reversal procedure presented in Section~\ref{sec:naive-alg}, with the difference that the process $Z_0$ is now initialized from a Gaussian distribution around $x_f$, i.e. $\mathcal N(\cdot;x_f,\sigma^2)$ with $\sigma>0$. 
We prove that, with a small enough $\sigma$, the resulting control law is able to solve problem~\ref{prob:approx-steer-to-y} while maintaining non-singularity, in the linear Gaussian setting,  with error bounds established in Proposition~\ref{prop:terminal-error},
while deferring the analysis of error bounds for the general nonlinear case to future work.


\subsection{Analysis of the linear Gaussian setting}\label{sec:lin-gau}
In this section, we study the proposed time-reversal method in the linear Gaussian setting, where
\begin{align*}
    f(x) = Ax,\quad g(x) = B,
\end{align*}
for matrices  $A \in \R^{n\times n}$ and $B \in \R^{n \times m}$. In this case, $f$ is linear and $g$ is constant, implying that both functions are smooth and globally Lipschitz. Therefore, Assumption~\ref{asum-fg} is satisfied. Furthermore, Assumption \ref{asum-brackets} will also hold whenever $(A,B)$ is controllable.

Under such setting, the SDEs \eqref{eq:Z-alg}, \eqref{eq:tilde-Z-alg}, and ~\eqref{eq:X-alg} take the form
\begin{subequations}
\begin{align}
    &\ud Z_t = -A Z_t \ud t + \epsilon B \ud W_t, \hspace{59pt} Z_0 = x_f,\label{eq:Z-lin}\\
    &\ud \tilde Z_t = A \tilde Z_t \ud t + B( k(t,\tilde Z_t) \ud t +\epsilon \ud \tilde W_t), \hspace{2pt} \tilde Z_T = x_f, \label{eq:tilde-Z-lin}\\
    &\ud X_t \!=\! A X_t \ud t \!+\! B(k(t,X_t)\ud t \!+ \!\epsilon \ud W_t), \hspace{10pt} X_0\! =\! x_0. \label{eq:X-lin}
\end{align}
\end{subequations}
The probability density of $Z_t$ is Gaussian, for all $t>0$, because \eqref{eq:Z-lin} is linear and the initial state $Z_0$ is deterministic.  In particular, $p(t,x)=\mathcal N(x;m_t,\Sigma_t)$ where the mean and covariance are given by
\begin{align*}
      m_t &= e^{-At}x_f,\\
      \Sigma_t &=  \epsilon^2 \int_0^t e^{-A(t-s)}BB^\top e^{-A^\top(t-s)}\ud s.
\end{align*}
Substitution of the Gaussian distribution formula in \eqref{eq:k-star}, and the use of \eqref{eq:fbcontrol}, yields the following formula for the feedback control law:
\begin{align}\label{eq:k-linear}
    k(t,x) 
    &= -\epsilon^2 B^\top \Sigma_{T-t}^{-1}(x-m_{T-t}),
\end{align}
for $t \in [0,T)$. This is similar to the control law that appears in~\cite[Eq. (11)]{pakniyat2021steering} and \cite{chen2015stochastic}, obtained through a stochastic optimal control formulation.

It is worth remarking that in the limit as $t \to T$, the covariance $\Sigma_{T-t} \to 0$, thus the control law becomes singular, as described in Section~\ref{sec:singular}. To resolve the singularity issue, we initialize $Z_0 \sim \mathcal{N}(\cdot;x_f,\sigma^2I)$. Under this initial condition, the distribution of $Z_t$ remains Gaussian $\mathcal N(x;m_t,Q_t)$ with the same mean  as before, but with a new covariance that is given by 
$
    Q_t = \Sigma_t + \sigma^2 e^{-At}e^{-A^\top t}
$.
The resulting feedback control law is
\begin{align}\label{eq:approx-k-linear}
    \hat k(t,x) 
    &= -\epsilon^2 B^\top Q_{T-t}^{-1}(x-m_{T-t}).
\end{align}
The new feedback control law remains nonsingular as $t\to T$. However, there is no guarantee that it would steer $X_t$ to the target state $x_f$. The following proposition  characterizes the error $\mathbb E[\|X_T-x_f\|^2]$ when the control law~\eqref{eq:approx-k-linear} is used instead of~\eqref{eq:k-linear}. 
\begin{proposition}\label{prop:terminal-error}
    Under the feedback control law defined in \eqref{eq:approx-k-linear}, the expected squared error between the terminal state and the target state is given by
    \begin{equation}\label{eq:error}
        \Expect[\|X_T-x_f\|^2] = \sigma^4\|e^{TA}x_0-x_f\|^2_{M^{-2}} + \sigma^2(n -\trace(M^{-1})),
    \end{equation}
    where $M := \sigma^2 I +  \epsilon^2\int_0^T e^{As}BB^\top e^{A^\top s}\ud s$. Moreover, for any $\delta>0$, there exists a small enough $\sigma>0$ that solves Problem~\ref{prob:approx-steer-to-y}.
\end{proposition}
\begin{proof}
     Let $\kappa$ and $\tilde \kappa$ denote the probability transition kernels associated with SDEs~\eqref{eq:Z-lin} and~\eqref{eq:tilde-Z-lin}, respectively, similar to the proof of Proposition~\ref{prop:sol2p1}. In terms of the kernels, the probability distribution of $X_T$ is equal to $\tilde \kappa_{T,0}(\cdot|x_0)$, because the SDE~\eqref{eq:X-lin} and~\eqref{eq:tilde-Z-lin} have the same form and $X_0=x_0$.  Then, upon the application of the time-reversal relationship~\eqref{eq:kappa-identity}, and the fact that $p(t,x) = \mathcal N(x;m_t,Q_t)$,
\begin{align*}
    \tilde \kappa_{T,0}(x|x_0) &=  \frac{ \kappa_{T,0}(x_0|x)p(0,x)}{p(T,x_0)} \\
    & = \frac{\mathcal N(x_0;e^{-AT}x,\Sigma_T) \mathcal N(x;x_f,\sigma^2I)}{\mathcal N(x_0;m_T,Q_T)}\\&=\mathcal N(x;\mu,P)
\end{align*}
where 
\begin{align*}
    \mu &= x_f + \sigma^2 M^{-1}e^{TA}(x_0 - e^{-TA}x_f),\\
    P & =  \sigma^2(I- \sigma^2 M^{-1}).
\end{align*}
  Now, to compute the error   $\mathbb E[\|X_T-x_f\|^2]$, we use the fact that $X_T \sim \mathcal N(\cdot;\mu, P)$ and, hence: 
\begin{align}
    \mathbb E[\|X_T-x_f\|^2] &= \|\mu-x_f\|^2 + \trace(P) \nonumber
\end{align}
which yields~\eqref{eq:error}. Moreover, in the limit as $\sigma \to 0$, $M$~converges to the controllability grammarian matrix which is positive definite  under the controllability assumption.  Therefore, taking the limit of~\eqref{eq:error} as $\sigma \to 0$, concludes that the error converges to zero, implying that for any $\delta>0$, there exists a $\sigma>0$ such that the error is smaller than $\delta$. 
\end{proof}

\begin{remark}\label{rem:improve-alg}
    The error \eqref{eq:error} comprises both a bias term and a variance term. The variance term is independent of $x_0$ and $x_f$, and bounded by $\sigma^2n$. In contrast, the bias term can be significant when $e^{TA}x_0$ and $x_f$ are far apart. This term arises due to the difference between the mean $\mathbb E[Z_T]=e^{-TA}x_f$ and $x_0$. To address this issue, we allow our proposed methodology the flexibility to incorporate a deterministic control input $u_t$ when $Z_t$ is simulated, and modify the control law to  $U_t = k(t,X_t) + \tilde u_t$ with $\tilde u_t = -u_{T-t}$\footnote{Our method can also incorporate a feedback control law in addition to the deterministic input. We do not pursue this in the paper, as computing a feedback law is considerably more challenging than finding a deterministic input.}. The addition of the deterministic input decreases the bias error by bringing the mean of $Z_T$ and $x_0$ closer. Infact, the difference becomes zero in the linear case when $(A,B)$ is controllable.  The details of this modification to the algorithm appears in~\ref{sec:Algorithm}. 
    
    
\end{remark}
\begin{algorithm}[bt]
\caption{Time-Reversal Control Synthesis} 
\begin{algorithmic}[1]
\STATE \textbf{Input:}  sample size $N$, step-size $\Delta t$, variance $\sigma$, deterministic control input $u_t$, function class $\Psi$.   
\STATE  $\{Z_0^i\}_{i=1}^N \sim \mathcal{N}(x_f, \sigma^2I)$

\FOR{$t\in \{\Delta t, 2\Delta t,\ldots,T-\Delta t,T \}$}
\STATE $\{\Delta  W^i_t\}_{i=1}^N\sim N(0,\Delta t I_n)$
	\STATE $ Z^{i}_{t+\Delta t} \!=\!  Z^i_{t}\!+\!(-f(Z^{i}_{t})\!+\!\epsilon^2\mathfrak g( Z^{i}_{t})+g(Z^{i}_{t})u_t)\Delta t  + \epsilon g(Z^{i}_{t})  \Delta  W^i_t$
\ENDFOR
\STATE $k^\star(t,\cdot)\!=\! \argmin_{k\in \Psi} \frac{1}{N}\sum_{i=1}^N \big[\frac{1}{2}\|g( Z^{i}_{t})k(t,Z^i_{t}) \!+\! \mathfrak g( Z^i_{t})\|^2 \!+\!  \sum_{j,l}^n(G_{j,l}(Z^i_{t})\partial_{x_l} (g(Z^{i}_{t})k(t,Z^i_{t})\!+\!\mathfrak g( Z^i_{t}))_j)\big]$

 \STATE \textbf{Output:}  $\{k^\star(t,\cdot)\}_{t\in \{0,\Delta t,\ldots,T\} }$
\end{algorithmic}
\label{alg:improved}
\end{algorithm}

\section{Numerical Results}\label{sec:results}
\subsection{Numerical Algorithm}
\label{sec:Algorithm}
In this section, we introduce our proposed numerical algorithm which is based on the methodology described in Section~\ref{sec:methodology}. The algorithm starts with simulating $N$ random realizations $\{Z^i_t\}_{i=1}^N$ of the process \eqref{eq:Z-alg}, with the flexibility of considering an additional  deterministic control input $u_t$. The simulation is carried out using the Euler-Maruyama discretization method with the step-size $\Delta t$. In order to find the control law~\eqref{eq:fbcontrol}, we use the decomposition  of the score function \eqref{eq:score-decom} and modify the score function optimization problem \eqref{eq:score-opt} according to $\min_k\,J(gk+\mathfrak g)$. 
To solve this optimization problem, we parameterize $k$ with a neural network with a 3-block ResNet architecture where each block consists of 2 linear layers of width 32 and an exponential linear unit (ELU)-type activation function. We use ADAM optimizer to find the parameters of the neural network.
The batch is generated by uniformly sampling $K_1=\lfloor\frac{T}{10\Delta t}\rfloor + 1$ time instants $\{t_1,t_2,\ldots,t_{K_1}\}$  from $[0,T]$, and $K_2=32$ random samples of the $N$ trajectories $\{Z^i_{t_1},\ldots,Z^i_{t_{K_1}}\}_{i=1}^N$. The details of the algorithm appear in Algorithm \ref{alg:improved}.
The code for reproducing the results is available online~\footnote{\url{https://github.com/YuhangMeiUW/P2P}}.

The deterministic control input $u_t$ is designed   to approximately bring $Z_T$ to the vicinity of  $x_0$. For example, this control may be obtained by the application of trajectory optimization techniques to the deterministic version of the model~\eqref{eq:Z-alg}, when $\epsilon =0$.   In addition to decreasing the bias error, the control input $u_t$ serves as an ``importance sampling" mechanism that guides $Z_t$ to be sampled in areas of the domain where the control law is more relevant to the initial condition $x_0$, thus increasing sampling efficiency.

\subsection{Two-dimensional Brownian Bridge}\label{sec:bb-exp}
We consider a 2-dimensional system governed by the SDE
\begin{equation}\label{eq:2nd-bb}
    \ud X_t = U_t \ud t + \epsilon \ud W_t, ~X_0=x_0,
\end{equation}
and let $\epsilon=0.3$, $x_0 = (0,0)^\top$, $T=1$, and $x_f = (2,2)^\top$. This is a linear Gaussian model with $A=0$ and $B=I$. We employ the deterministic control $u_t=x_0-x_f$ which brings $Z_0=x_f$ to $Z_1=x_0$ in the deterministic setting. In this case, the resulting control law for $X_t$ takes the form 
\begin{align}
    U_t=  x_f-x_0 - \frac{\epsilon^2(X_t-(1-t)x_0-tx_f)}{\epsilon^2(1-t) + \sigma^2}.\label{eq:analy-sol-u}
\end{align}

We apply Algorithm \ref{alg:improved} with $u_t = x_0-x_f$, $N=1000$, $\sigma=0$, and $\Delta t = 0.004$. The resulting  trajectories $\{Z^i_t\}_{i=1}^N$ and $\{X^i_t\}_{i=1}^N$, along with the control inputs $\{U^i_t\}_{i=1}^N$,  are shown in Fig. \ref{fig:bb-X-compare}. 
The result demonstrates that the control law obtained from Algorithm \ref{alg:improved} successfully steers all trajectories $X^i_t$ from $x_0$ to desired target $x_f$.
\begin{figure*}[tb]
	\centering
    \begin{subfigure}{0.32\textwidth}\centering
        \includegraphics[width=\hsize,trim={0pt 0pt 12pt 41pt},clip]{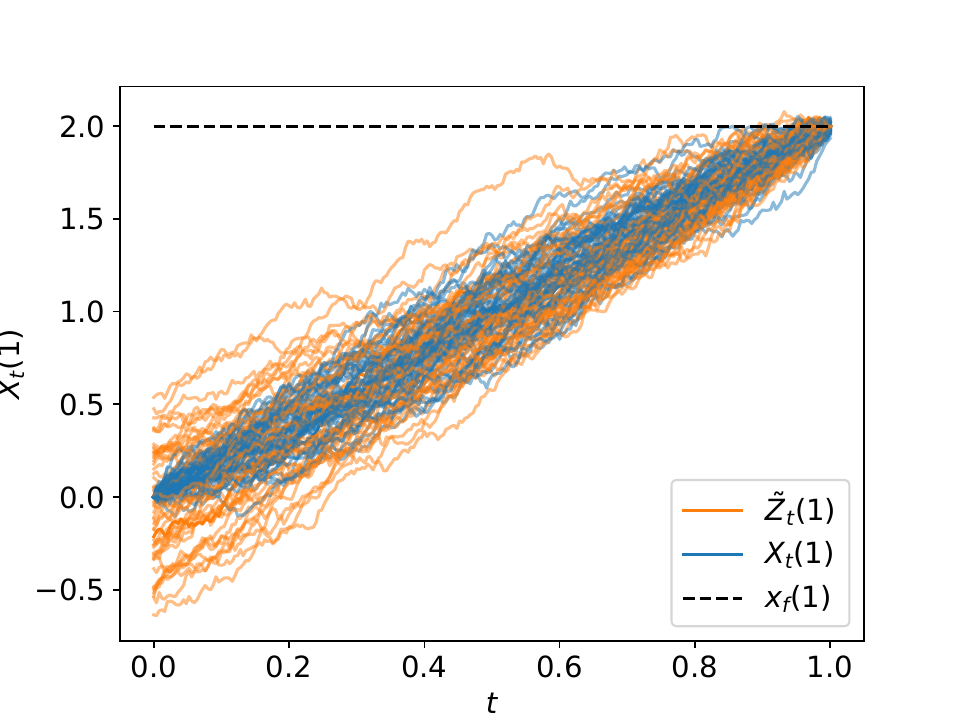}
        \caption{}
        \label{fig:xz1}
    \end{subfigure}
    \begin{subfigure}{0.32\textwidth}
        \includegraphics[width=\hsize,trim={0pt -5pt 12pt 41 pt},clip]{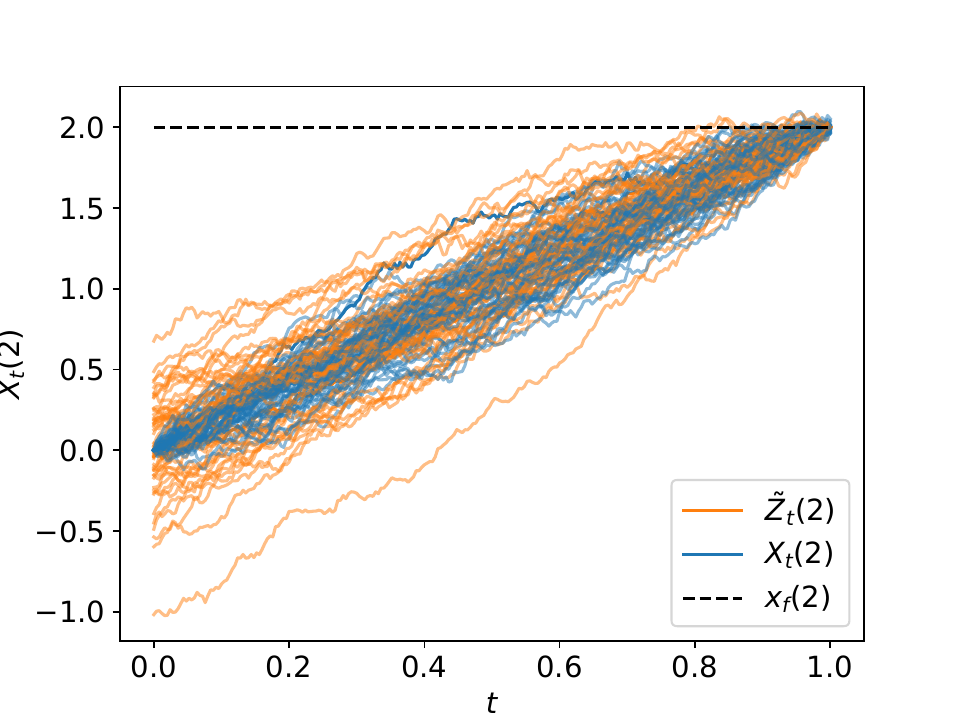}
        \caption{}
        \label{fig:xz2}
    \end{subfigure}
    \begin{subfigure}{0.32\textwidth}
        \includegraphics[width=\hsize,trim={0pt 0pt 12pt 41 pt},clip]{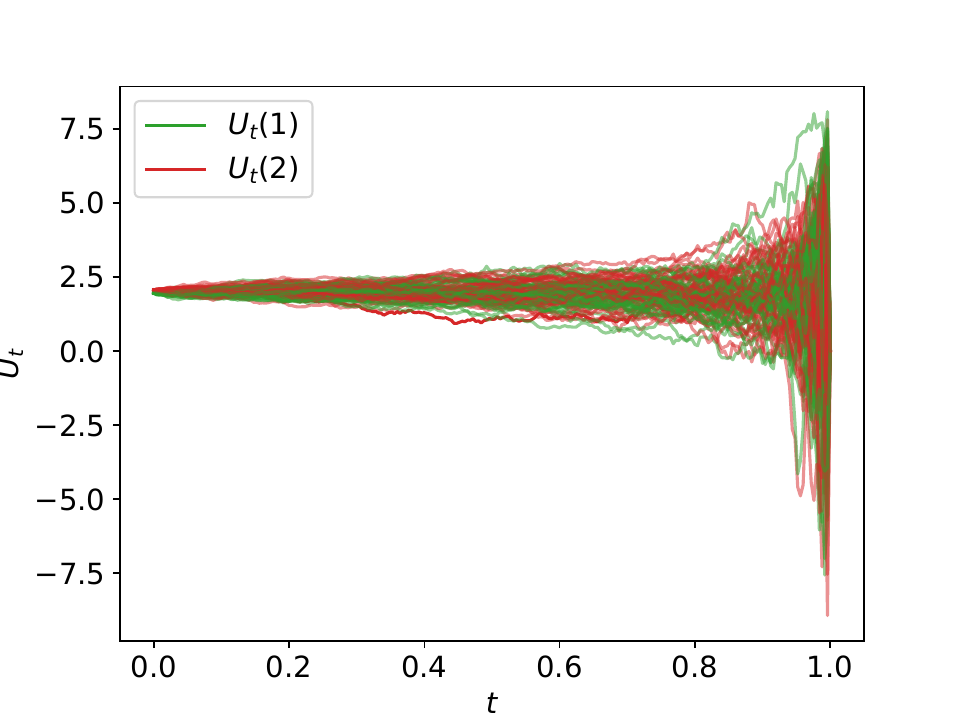}
        \caption{}
        \label{fig:u}
    \end{subfigure}
	\caption{Numerical result for the application Algorithm~\ref{alg:improved} to the two-dimensional Brownian bridge example of Section~\ref{sec:bb-exp}: (a) First component of $X_t$ and $\tilde Z_t$ (b) Second component of $X_t$ and $\tilde Z_t$ (c) Control input $U_t$.
    }
    \label{fig:bb-X-compare}
\end{figure*}

In order to quantify the performance of the control law, we introduce the following mean-squared-error (MSE) criteria 
\begin{equation}\label{eq:mse}
    MSE = \frac{1}{N}\sum_{i=1}^N \|X_T^i - x_f\|_2^2.
\end{equation}
We investigate the influence of the time step $\Delta t$ and standard deviation $\sigma$ on the MSE criteria. 
The result for varying time step-size is presented in Fig \ref{fig:mse_dt}, where we fix $\sigma=0$. We also show the MSE corresponding to implementing the exact form of the control law~\eqref{eq:analy-sol-u} and the open-loop control $U_t=x_f-x_0$ as baselines. It is observed that the algorithm performs almost as good as the exact solution, and the MSE decreases as $\Delta t \to 0$. The result for varying $\sigma$ is presented in Fig \ref{fig:mse_sigma} with fixed $\Delta t = 0.004$, where for comparison, the case without using the deterministic input is also included. It is observed that including the deterministic input significantly decreases the MSE when $\sigma>0$, justifying the remark~\ref{rem:improve-alg}. Moreover, $\sigma$ acts as a regularizer in the optimization procedure and decreases the difference between Algorithm 1 and the exact solution. Although increasing $\sigma$ increases MSE, it serves to avoid singularity of the control,  which is shown by computing  the averaged control energy 
\begin{equation}\label{eq:unorm}
    U_{norm} = \frac{1}{N}\sum_{i=1}^N\int_0^T \|U_t^i\|^2 \ud t
\end{equation}
as a function of $\sigma$, in Fig \ref{fig:unorm}. The MSE and  $U_{norm}$ are both averaged over $5$ independent experiments, where the shaded region represents the range from the minimum to the maximum across experiments.
\begin{figure*}[tb]
	\centering
    
    \begin{subfigure}{0.32\textwidth}\centering
        \includegraphics[width=\hsize,trim={0pt 0pt 12pt 41pt},clip]{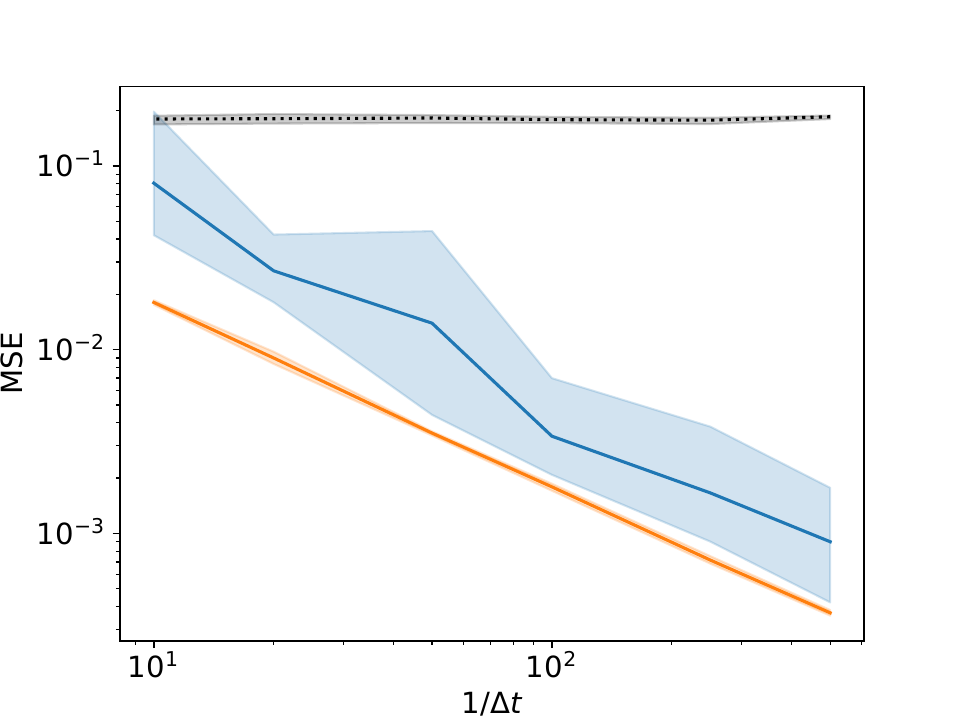}
        \caption{}
        \label{fig:mse_dt}
    \end{subfigure}
    \begin{subfigure}{0.32\textwidth}
        \includegraphics[width=\hsize,trim={0pt -5pt 12pt 41 pt},clip]{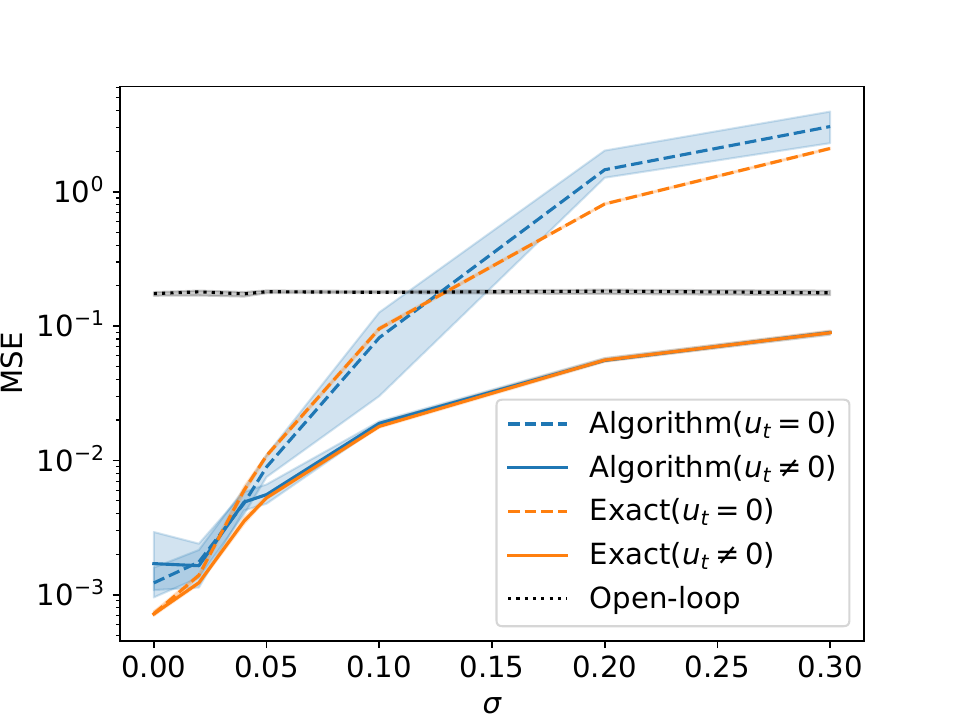}
        \caption{}
        \label{fig:mse_sigma}
    \end{subfigure}
    \begin{subfigure}{0.32\textwidth}
        \includegraphics[width=\hsize,trim={0pt 0pt 12pt 41 pt},clip]{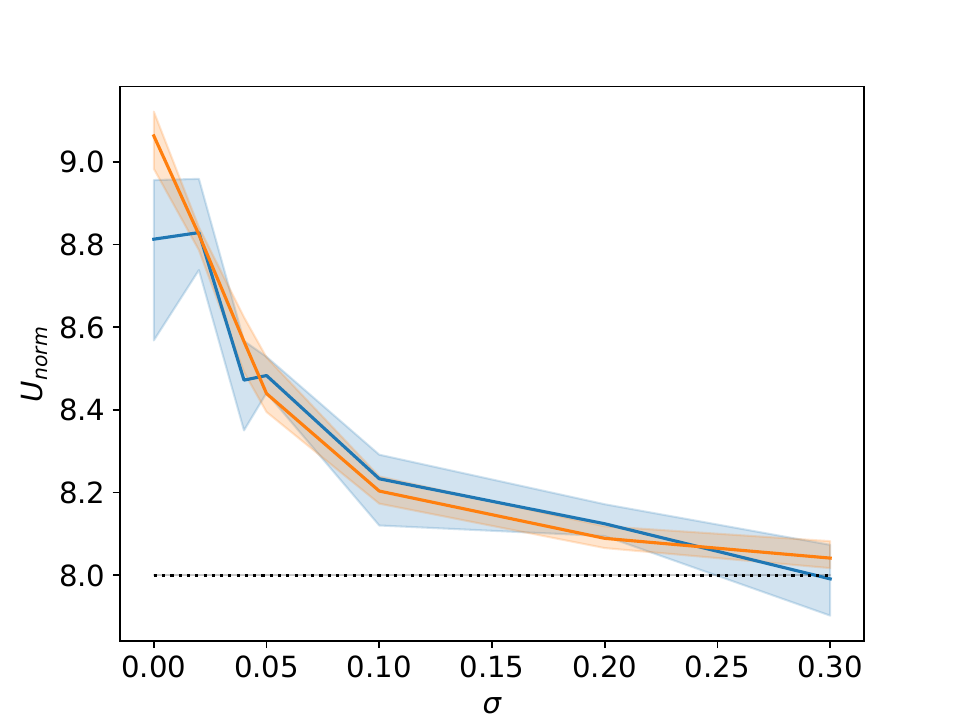}
        \caption{}
        \label{fig:unorm}
    \end{subfigure}
	\caption{Numerical error analysis for the application Algorithm~\ref{alg:improved} to the two-dimensional Brownian bridge example of Section~\ref{sec:bb-exp}:  (a) Influence of time step-size on MSE~\eqref{eq:mse}, with $\sigma=0$ (b) Influence of $\sigma$ on MSE, with $\Delta t = 0.004$ (c) Influence of $\sigma$ on $U_{norm}$~\eqref{eq:unorm}, with $\Delta t = 0.004$. The results compare (i) Algorithm \ref{alg:improved}; (ii) the exact solution~\eqref{eq:analy-sol-u}; and  implementing the open loop control $U_t=x_f-x_0$. Panel (b) also includes the case where the deterministic input $u_t=0$ in both exact solution and Algorithm~\ref{alg:improved}.}
\end{figure*}
\subsection{Inverted Pendulum}\label{sec:IP-exp}
We consider the  stochastic pendulum dynamics with 
\begin{align*}
    f(x) =  \begin{bmatrix}
        x(2) \\
        \sin(x(1)) - 0.01x(2)
    \end{bmatrix},\quad g(x) = \begin{bmatrix}
        0\\1
    \end{bmatrix},
\end{align*}
$\epsilon=0.3$, $x_0=[\pi,0]^\top$, $x_f = [0,0]^\top$, and $T=5$. The first component of the state is the angle where the angle equal to $0$ denotes the upward position of the pendulum and $\pi$ denote the downward. The goal is to bring the pendulum from the downward position to the upward position. We apply Algorithm~\ref{alg:improved} with $N=1000$, $\sigma=0$, $u_t=0$, and $\Delta t =0.004$. The resulting trajectories and control inputs are shown in Fig. \ref{fig:ip-traj}, where the successful steering of the pendulum to the upward position is demonstrated. 

\begin{figure*}[tb]
	\centering
    \begin{subfigure}{0.32\textwidth}\centering
        \includegraphics[width=\hsize,trim={0pt 0pt 12pt 41pt},clip]{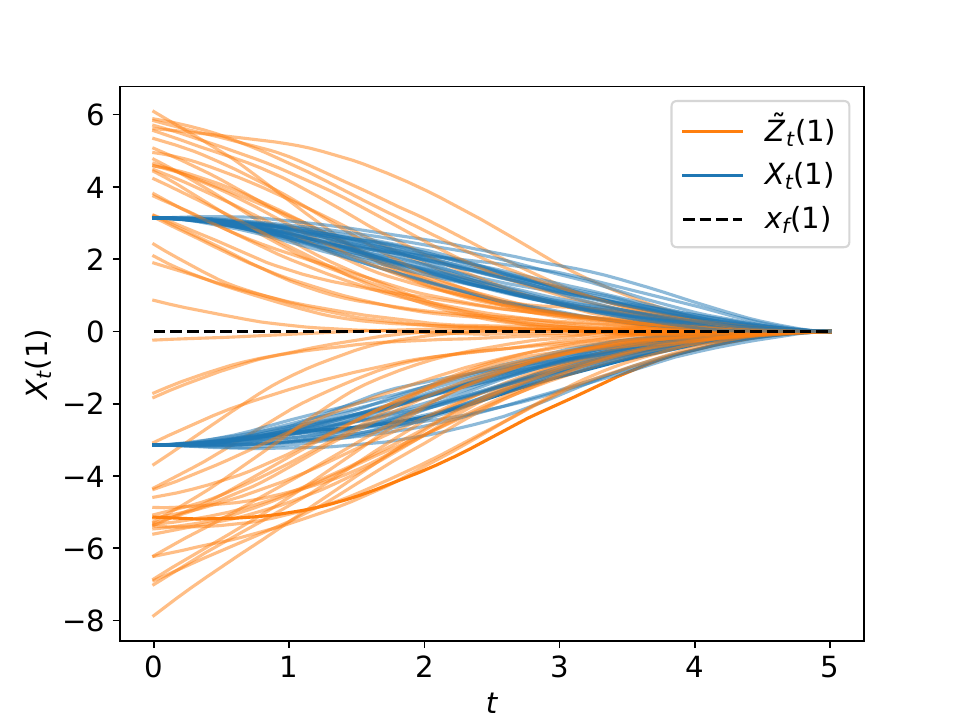}
        \caption{}
        \label{fig:ip-xz1}
    \end{subfigure}
    \begin{subfigure}{0.32\textwidth}
        \includegraphics[width=\hsize,trim={0pt -5pt 12pt 41 pt},clip]{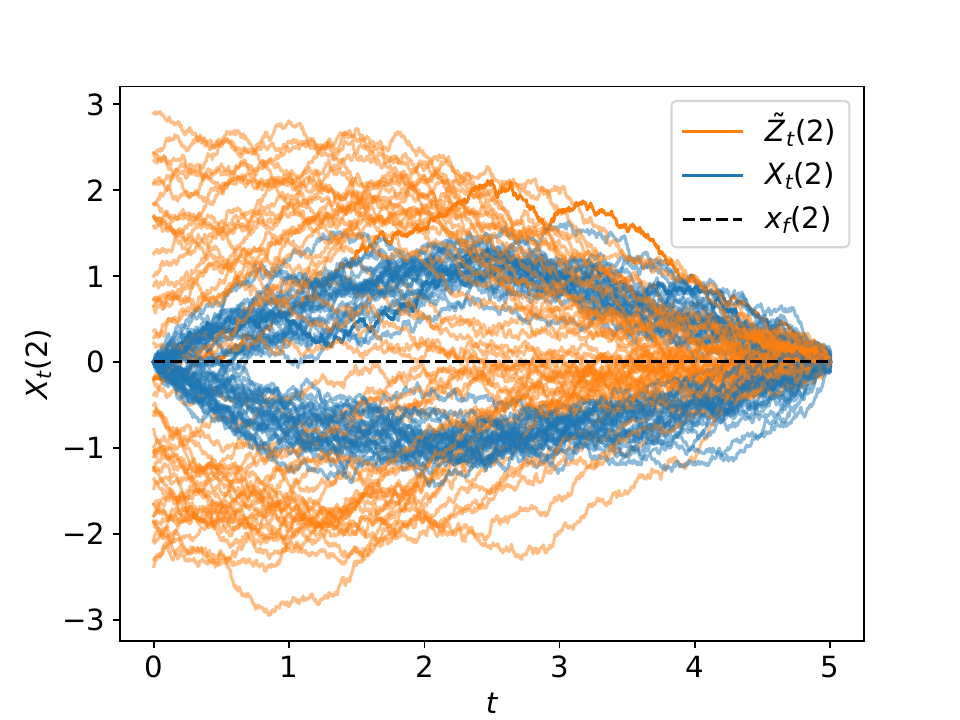}
        \caption{}
        \label{fig:ip-xz2}
    \end{subfigure}
    \begin{subfigure}{0.32\textwidth}
        \includegraphics[width=\hsize,trim={0pt 0pt 12pt 41 pt},clip]{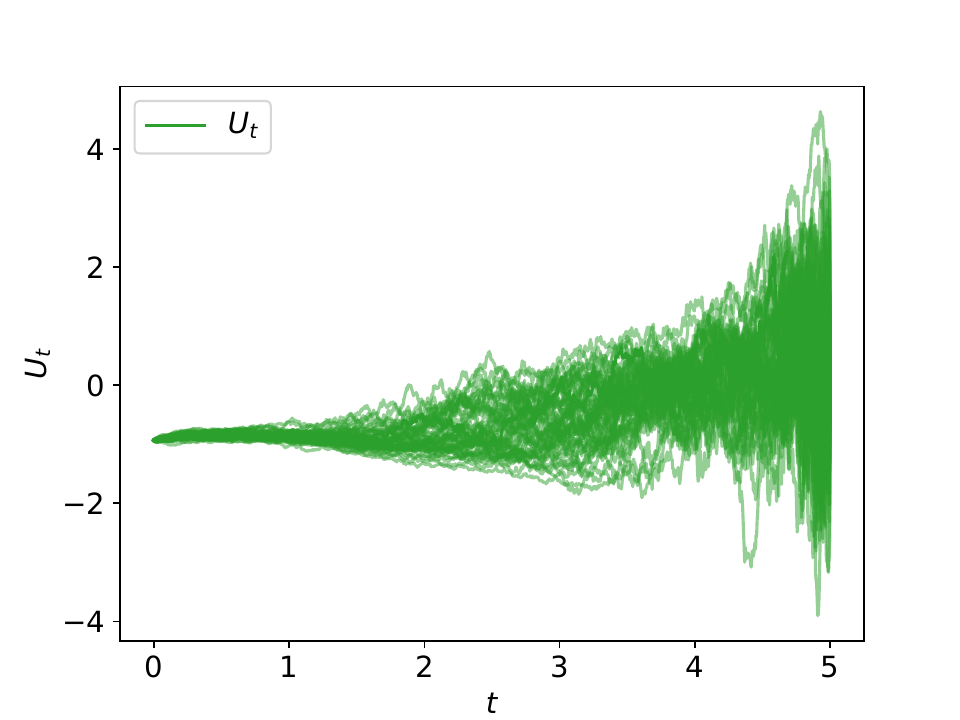}
        \caption{}
        \label{fig:ip-u}
    \end{subfigure}
	\caption{Numerical result for the application Algorithm~\ref{alg:improved} to the inverted pendulum example of Section~\ref{sec:IP-exp}: (a) First component of $X_t$ and $\tilde Z_t$ ; (b) Second component of $X_t$ and $\tilde Z_t$; (c) Control input $U_t$. The first component of the initial state  is equal to $\pm\pi$ and represents the downward position of the pendulum, while the first component of the terminal state is equal to $0$, representing the upward position. 
    }
    \label{fig:ip-traj}
\end{figure*}

{
\section{Conclusion}

This paper introduces a novel approach for steering nonlinear stochastic control-affine systems to a desired target state within a finite time horizon, leveraging time-reversal theory of diffusions. By constructing feedback control laws based on the score function associated with the reversed dynamics, the proposed method ensures finite-time convergence to the target state. Unlike traditional Schr\"{o}dinger bridge methods or stochastic optimal control formulations, our approach is computationally efficient and applicable to both linear and nonlinear stochastic systems without relying on optimality relative to the initial condition.

An extension of the theory to address practical challenges related to inevitable singularities in the control law near the terminal time is also presented through relaxation of the almost-sure constraint to a distribution constraint and explicit error bounds of this approach are provided for the linear Gaussian case. Numerical experiments demonstrate the effectiveness of the method across benchmark examples, including a Brownian bridge and inverted pendulum dynamics. 

Future research includes the extension of the theoretical results for assigning terminal distribution constraints in the general nonlinear system setting, the improving score function approximation techniques,  exploring the optimality properties of the proposed control law through its connection to Doob's $h$-transform, and numerical verification of our method on real-world control applications in Aerospace~\cite{exarchos2019optimal,lew2020chance}.

}


    \bibliographystyle{IEEEtran}
    \bibliography{references}

\end{document}